\documentclass[11pt,british,reqno]{amsart}
\usepackage[T1]{fontenc}
\usepackage[latin1]{inputenc}
\usepackage{a4wide}
\usepackage{graphicx}
\usepackage{setspace}
\makeatletter
 \theoremstyle{plain}
 \newtheorem{thm}{Theorem}[section]
 \numberwithin{equation}{section} 
 \numberwithin{figure}{section} 
 \theoremstyle{plain}
 \theoremstyle{remark}
 \newtheorem{rem}[thm]{Remark}
 \theoremstyle{plain}
 \newtheorem{prop}[thm]{Proposition} 
 \theoremstyle{plain}
 \newtheorem{cor}[thm]{Corollary} 
 \theoremstyle{plain}
 \newtheorem{lem}[thm]{Lemma} 

\usepackage{amssymb}
\usepackage{setspace}
\usepackage{float}
 
\usepackage{graphicx}
\usepackage{mathptmx}
\usepackage{psfrag}

\usepackage{amscd}

\newcommand{\1}{\mathbf 1}

\renewcommand{\hat}{\widehat}

\theoremstyle{remark}
 \newtheorem*{conclusion*}{Conclusion}
\newtheorem*{remark*}{Remark}

\def\R{{\mathbb R}}
 \def\N{{\mathbb N}}

\usepackage{babel}
\makeatother
\begin{document}

\title{Fractal analysis for sets of non-differentiability of Minkowski's
question mark function}

\date{\today{}}

\author{Marc Kesseb\"{o}hmer and Bernd O. Stratmann}

\address{Fachbereich 3 - Mathematik und Informatik, Universit\"{a}t Bremen,
D--28359 Bremen, Germany}

\email{mhk@math.uni-bremen.de}

\address{Mathematical Institute, University of St Andrews, St Andrews KY16
9SS, Scotland}
\email{bos@maths.st-and.ac.uk}

\begin{abstract}
In this paper we study various fractal geometric aspects of the Minkowski
question mark function $Q.$ We show that  the unit interval can be written as
the union of the three sets $\Lambda_{0}:=\{ x:Q'(x)=0\}$, 
$\Lambda_{\infty}:=\{ x:Q'(x)=\infty\}$, and $\Lambda_{\sim}:=\{ x:Q'(x)\textrm{
does not exist and } Q'(x)\neq\infty\}.$  The main result is that the Hausdorff
dimensions of these sets are related in the following way. \[
\dim_{H}(\nu_{F})<\dim_{H}(\Lambda_{\sim})= \dim_{H} \left( \Lambda_{\infty}
\right) = \dim_{H} \left(\mathcal{L}(h_{\mathrm{top}})\right)<\dim_{H}\left(\Lambda_{0}\right)=1.\]
 Here, $\mathcal{L}(h_{\mathrm{top}})$ refers to the level set of the
 Stern-Brocot multifractal decomposition at the topological entropy
 $h_{\mathrm{top}}=\log2$ of the Farey map $F,$ and $\dim_{H}(\nu_{F})$ denotes
 the Hausdorff dimension of the measure of maximal entropy of the dynamical system associated with $F.$ The proofs rely
partially on the multifractal formalism for Stern-Brocot intervals and give
non-trivial applications of this formalism.

\end{abstract}
\maketitle
\let\languagename\relax

\section{Introduction and statement of results}

In this paper we return to the origins of the multifractal analysis
of measures, which started with work on fractal sets by Mandelbrot
and others in the 1980s (see e.g. \cite{GM}, \cite{Halsey}, 
\cite{Mandelbrot}, \cite{Frisch}).
For this, we go even further back in time, and consider a function
$Q$ of the unit interval $\mathcal{U}$ into itself, which was originally
designed by Minkowski \cite{Min:04} in order to illustrate the Lagrange
property of quadratic surds. Today, this function is usually referred
to as the Minkowski question mark function, and it appears in various
different disguises. For instance, it appears as the distribution
function of the measure of maximal entropy $\nu_{F}$ for the dynamical
system arising from the Farey map $F$. That is, \[
Q(x)=\nu_{F}([0,x)),\,\,\,\textrm{for all }\,\, x\in\mathcal{U}.\]
 Since the support of $\nu_{F}$ is equal to $\mathcal{U}$, and since
$\nu_{F}$ is singular with respect to the $1$-dimensional Lebesgue
measure $\lambda$ on $\mathcal{U}$ (see Salem \cite{Salem}), the
graph of $Q$ is appropriately described by the term `slippery devil's
staircase', a term which was coined by Gutzwiller and Mandelbrot in
\cite{GM} (see also \cite{G}, \cite{B}). Another disguise of $Q$ is, that
it provides a stable bridge between the Farey system and the binary
system $(\mathcal{U},T)$, that is the dynamical system which arises
from the tent map $T$. In this disguise, the homeomorphism
$Q$ represents the topological conjugacy map between the Farey system
and the tent system, such that $T\circ Q=Q\circ F$. Using elementary
observations for the regular continued fraction expansion $x=[a_{1},a_{2},\ldots]$
of elements $x\in\mathcal{U}$, one readily rediscovers the following
alternating sum representation of $Q$, first obtained by Denjoy \cite{Denjoy0}
(see also \cite{Denjoy} , \cite{Salem}, \cite{R1}, \cite{R2}),
\[
Q(x):=-2\sum_{k\in\N}(-1)^{k}\,\,2^{-\sum_{i=1}^{k}a_{i}},\,\,
\textrm{for all }\,\, x=[a_{1},a_{2},\ldots]\in\mathcal{U}.\]
 These observations mark the starting point for the fractal geometric
analysis of the function $Q$ in this paper. We will show that interesting
measure theoretical aspects of the Minkowski scenario can be derived
from the recently obtained multifractal analysis for Stern-Brocot
intervals \cite{KesseboehmerStratmann:07}. As a first demonstration
of the fruitfulness of this approach, we study fractal geometric
relationships
between $Q$, $\nu_{F}$ and the Gauss measure $m_{G}$. We obtain
the result
that one can explicitly compute the integral over $Q$ with respect
to $m_{G}$, as well as the integral with respect to $\nu_{F}$ over
the distribution function $\Delta_{m_{G}}$ of $m_{G}$. That is,
with $\dim_{H}$ referring to the Hausdorff dimension, we obtain 
\[
\int_{\mathcal{U}}Q\, dm_{G}=1-\int_{\mathcal{U}}\Delta_{m_{G}}\, d\nu_{F}
=(\dim_{H}(\nu_{F})-1/2)/\dim_{H}(\nu_{F})\,\,\,(\approx0.571612).\]
 As an immediate consequence of this, one can then also rediscover
a result by Kinney \cite{Kinney} which expresses the Hausdorff dimension
of $\nu_{F}$ in terms of a certain explicit integral. \\
 Subsequently, we draw the attention to the derivative $Q'$ of $Q$.
It was shown only relatively recently in \cite{PVB:01} that if $Q'(x)$
exists in the generalised sense, meaning that $Q'(x)$ either exists
or is equal to infinity, then $Q'(x)$ either vanishes or else is
equal to infinity. We give a new and very elementary proof of this
fact, and then add to this by showing that $Q'(x)$ is equal to infinity
if and only if $\lim_{n\rightarrow\infty}\nu_{F}(T_{n}(x))/\lambda(T_{n}(x))$
is equal to infinity. Here, $T_{n}(x)$ refers to the unique atom
of the $n$-th refinement of $\mathcal{U}$ with respect to $F$,
such that $x\in T_{n}(x)$. Moreover, we show that 
if  for the approximants $p_{k}/q_{k}$ of $x =[a_{1},a_{2},\ldots]$ 
we have $\lim_{k\rightarrow\infty} a_{k+1} \cdot  \nu_{F} 
\left([p_{k}/q_{k},p_{k+1}/
q_{k+1})_{\pm}\right) \, / \, \lambda\left([p_{k}/q_{k},p_{k+1}/
q_{k+1})_{\pm}\right)=0$, then
$Q'(x)$ vanishes (see Section 5 for the definition of $[\, \,  , \, \,
)_{\pm}$). The latter, slightly technical observations
will turn out to be crucial in the multifractal analysis of $Q'$
to come. In order to state the main results of this analysis, note
that $\mathcal{U}$ can be decomposed into mutually disjoint sets
as follows. \[
\mathcal{U}=\Lambda_{0}\cup\Lambda_{\infty}\cup\Lambda_{\sim},\]
 where $\Lambda_{0}:=\{ x:Q'(x)=0\}$, $\Lambda_{\infty}:=\{ x:Q'(x)=\infty\}$,
and $\Lambda_{\sim}$ refers to the set of elements for which $Q'$ does
not exist in the generalised sense. Surprisingly, before these 
investigations relatively  little was known about this decomposition. The main contributions
thus far were made by Salem, and these date back more than 60 years.
In our notation, the aforementioned result of Salem \cite{Salem}
reads as $\lambda(\Lambda_{0})=1$. More precisely, Salem \cite{Salem}
showed that if $Q'([a_{1},a_{2},\ldots])$ exists and is equal to
some finite value, and if, additionally, $\limsup_{n\rightarrow\infty}a_{n}=\infty$,
then $[a_{1},a_{2},\ldots]\in\Lambda_{0}$. The analysis in this paper
will give significant extensions of this classical result. In order
to state these extensions, recall that in \cite{KesseboehmerStratmann:07} we
computed the dimension spectrum of the multifractal decomposition
\[
\mathcal{L}(s):=\left\{ x\in\mathcal{U}:\lim_{n\rightarrow
\infty}\frac{\log\lambda(T_{n}(x))}{\log\nu_{F}(T_{n}(x))}=
\frac{s}{h_{\mathrm{top}}}\right\} .\]
Here, $h_{\mathrm{top}}=\log2$ refers to the topological entropy
of the Farey map $F$. In particular, in \cite{KesseboehmerStratmann:07}
it was shown that the Hausdorff dimension of $\mathcal{L}(s)$ is
nontrivial if and only if $s\in[0,2\log\gamma)$ (with $\gamma$ referring
to the Golden Mean). By relating this multifractal decomposition to
the Minkowski scenario in this paper, a first outcome is that \[
\mathcal{L}(s)\subset\Lambda_{\infty}
\textrm{\, for }s\in(h_{\mathrm{top}},2\log\gamma],\,
\textrm{whereas } \, \mathcal{L}(s)\subset\Lambda_{0}\: \, \textrm{for } \, s\in[0,h_{\mathrm{top}}).\]
 By expressing this result in terms of the convergents $p_{k}/q_{k}$
of elements $x=[a_{1},a_{2},\ldots]$, one then immediately derives
the following result. \[
\left\{ x:\lim_{n\rightarrow\infty}2\log q_{n}/\sum_{i=1}^{n}a_{i}>
h_{\mathrm{top}}\right\} \subset\Lambda_{\infty},\textrm{\, and }
\,\,\left\{x: \lim_{n\rightarrow\infty}2\log q_{n}/
\sum_{i=1}^{n}a_{i}<h_{\mathrm{top}}\right\} \subset\Lambda_{0}.\]
 Let us now finally come to the main result of this paper. For this,
note that on the basis of the results of Denjoy and Salem, one might suspect
that the complement of $\Lambda_{0}$ in $\mathcal{U}$ can still
be large, in the sense that its Hausdorff dimension could be
equal to one. Our main result now shows that this is in fact not the
case. More precisely, for the Hausdorff dimensions of $\Lambda_{\infty}$
and $\Lambda_{\sim}$, we obtain the result \[
0.875 \approx \dim_{H}(\nu_{F})<\dim_{H}(\Lambda_{\sim})=
\dim_{H}\left(\Lambda_{\infty}\right)=\dim_{H}
\left(\mathcal{L}(h_{\mathrm{top}})\right)<\dim_{H}\left(\Lambda_{0}\right)=1.\]
\begin{figure}
\psfrag{dimH}{\( \dim_H(\mathcal{L}(s))\)} \psfrag{dimH1}{ \(\dim_H(\Lambda_\infty)=
\dim_H(\Lambda_\sim)\)}\psfrag{dimH2}{ \(\dim_H(\nu_F)\approx 0.875 \)}
\psfrag{log2}{   \(\log 2\)}\psfrag{logG}{   \(2\log \gamma \)}
\psfrag{chi}{\(\chi_{\nu_F}\)}\psfrag{a}{\(s\)}\psfrag{1}{\(1\)}
\includegraphics[%
  width=0.70\columnwidth,
  keepaspectratio]{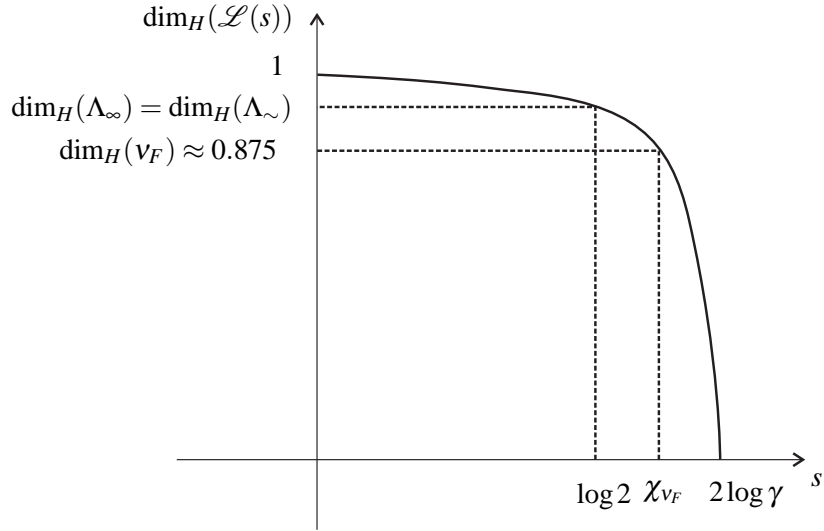}

\caption{\label{fig:Spectrum}The Stern-Brocot dimension spectrum}
\end{figure}
\noindent Here, the proof of the second equality $\dim_{H}\left(\Lambda_{\infty}\right)=\dim_{H}
\left(\mathcal{L}(h_{\mathrm{top}})\right)$ is derived from a 
non-trivial application of the multifractal formalism for Stern-Brocot 
intervals obtained in \cite{KesseboehmerStratmann:07}  (cf. Figure \ref{fig:Spectrum}), whereas the 
proof of the first 
equality $\dim_{H}(\Lambda_{\sim})=
\dim_{H}\left(\Lambda_{\infty}\right)$  combines this formalism with 
an extension of  the analysis of sets of `non-typical' points in 
\cite{BS00} to non-hyperbolic dynamical systems.

\begin{rem} In contrast to `ordinary devil's 
staircases', which usually arise from distribution
functions of fractal measures on Cantor-like sets, a  {\em slippery devil's staircase}
is the graph of the distribution function of a measure whose support 
is equal to the whole unit interval $\mathcal{U},$ but which is
nevertheless  singular with 
respect to the Lebesgue measure
$\lambda$ on $\mathcal{U}$.
Slippery devil's staircases should not be confused with
ordinary devil's staircases. In order to give
a brief demonstration of the difference between these two types of staircases,
let us consider the example of the homogeneous Cantor measure $\mu_{\mathcal{C}}$ 
supported on Cantor's ternary set $\mathcal{C}$.
It is immediately clear that the derivative of the  distribution
function $\Delta_{\mu_{\mathcal{C}}}$ 
vanishes on the complement of $\mathcal{C}$ in $\mathcal{U}$,
giving that $\lambda(\Lambda_{0}(\Delta_{\mu_{\mathcal{C}}}))=1$. 
 By a result of Darst \cite{Darst} (see also \cite{F}),
one has $\dim_{H}(\Lambda_{\sim}(\Delta_{\mu_{\mathcal{C}}}))=(\dim_{H}(\mathcal{C}))^{2}$.
Moreover, by a classical result of Gilman \cite{Gilman} we have that if the
derivative of $\Delta_{\mu_{\mathcal{C}}}$ exists in the generalised
sense at some point  $x\in \mathcal{C}$, then it can only be equal to   infinity.
Hence, $\dim_{H}(\Lambda_{\infty}(\Delta_{\mu_{\mathcal{C}}}))=
\dim_{H}(\mathcal{C})$.
Let us remark that the  result of Darst can be derived from straightforward
adaptations of techniques developed for estimating the Hausdorff dimension
of well-approximable irrational numbers (see e.g. \cite{Jarnik},
\cite{S}). Hence, in this situation, the set $\Lambda_{\sim}$
can be thought of as being conceptionally analogous to the set of
well-approximable numbers. This analogy no longer
holds for slippery devil's staircases.
\end{rem}

\section{Multifractal formalism for Stern-Brocot intervals revisited}

Let us first recall the classical construction of Stern-Brocot intervals
in the unit interval $\mathcal{U}:=[0,1]$ (\cite{Stern1858}, \cite{Brocot:1860},
see also \cite{HW}, \cite{Ito}, \cite{Richards}). For each
$n\in\N_{0}$, the elements of the Stern-Brocot sequence $\left\{ s_{n,k}/t_{n,k}:
k=0,\ldots,2^{n}\right\} $
of order $n$ are defined recursively for $n\in\N$,
$k=0,\ldots,2^{n-1}$
and $r=s,t$ as follows \[
s_{0,0}:=0,s_{0,1}:=t_{0,0}:=t_{0,1}:=1,r_{n,2k}:=r_{n-1,k}\,\,\textrm{and }
\,\, r_{n,2k-1}:=r_{n-1,k-1}+r_{n-1,k}.\]
 With this ordering of the rationals in $\mathcal{U}$ we define the
set $\mathcal{T}_{n}$ of Stern-Brocot intervals of order $n$ by
\[
\mathcal{T}_{n}:=\left\{ T_{n,k}:=\left[s_{n,k}/t_{n,k},s_{n,k+1}/t_{n,k+1}\right):
\, k=0,\ldots,2^{n}-1\right\} .\]
 Clearly, $\mathcal{T}_{n}$ is the set of atoms of the $n$-th
 refinement
of $\mathcal{U}$ with respect to the Farey map, and one immediately
finds that for each $x\in\mathcal{U}$ and $n\in\N_{0}$ there exists
a unique Stern-Brocot interval $T_{n}(x)\in\mathcal{T}_{n}$ such
that $x\in T_{n}(x)$. \\
 In \cite{KesseboehmerStratmann:07} (see also \cite{KesseboehmerStratmann:04A}
\cite{KesseboehmerStratmann:04B}), we considered the $n$-th Stern-Brocot
quotient $\ell_{n}$ and the Stern-Brocot growth rate $\ell$, which
are given by (assuming that the limit exists) \[
\ell_{n}(x):=\frac{1}{n}\,\,\log\left(1/\lambda\left(T_{n}(x)\right)\right)\,\,
\textrm{and }\,\,\ell(x):=\lim_{n\rightarrow\infty}\ell_{n}(x).\]
 Here, $\lambda$ refers to the $1$--dimensional Lebesgue measure
on $\mathcal{U}$.\\
 One of the main results in \cite{KesseboehmerStratmann:07} determined
the Lyapunov spectrum arising from $\ell$. That is, we computed the
Hausdorff $\dim_{H}$ of the following level sets \[
\mathcal{L}(s):=\left\{ x\in\mathcal{U}:\ell(x)=s\right\} ,
\textrm{ for }s\in\R.\]
 For the purposes of this paper the following main results
 of \cite{KesseboehmerStratmann:07}
will be crucial. Here, $P$ refers to the Stern-Brocot pressure function
$P$, which is given for $t\in\R$ by \begin{equation}\label{pressure}
P(t):=\lim_{n\rightarrow\infty}\frac{1}{n}\log\sum_{T\in\mathcal{T}_{n}}
\left(\lambda\left(T\right)\right)^{t},\end{equation}
 and $\widehat{P}$ refers to the Legendre transform, given for $s\in\R$
by $\widehat{P}(s):=\sup_{t\in\R}\{ st-P(t)\}$. Also, throughout
we let $\gamma:=(\sqrt{5}+1)/2$, and use the convention $\widehat{P}(0)/0:=-1$.

\vspace{2mm}

\begin{thm}[\cite{KesseboehmerStratmann:07}]
\label{KS} For each $s\in\left[0,2\log\gamma\right]$,
we have \[
\dim_{H}\left(\mathcal{L}(s)\right)=-\frac{\widehat{P}(-s)}{s}\,\left(=:d(s)\right).\]
 Here, the function $P$ has the following properties.
\begin{itemize}
\item $P$ is convex, non-increasing and differentiable throughout $\R$.
\item $P$ is real-analytic on the interval $(-\infty,1)$ and is equal
to $0$ on $[1,\infty)$.
\end{itemize}
Also, for the dimension function $d$ the following hold.
\begin{itemize}
\item $d$ is continuous and strictly decreasing on $[0,2\log\gamma]$,
and vanishes on $\R\setminus[0,2\log\gamma)$.
\item $d(0):=\lim_{t\searrow0}-\widehat{P}(-t)/t=1$, and $\lim_{t\nearrow2\log\gamma}d'\left(t\right)
=-\infty$.
\end{itemize}
\end{thm}

\section{Minkowski's question mark function}

In this section we will investigate the relationships between the following
two well known, elementary, measure theoretical dynamical systems.

\emph{The Farey-system} $\left(\mathcal{U},F,\nu_{F}\right)$: Let
$F:\mathcal{U}\rightarrow\mathcal{U}$ refer to the \emph{Farey-map}
on $\mathcal{U}$, given by \[
F(x):=\left\{ \begin{array}{lll}
x/(1-x) & \,\,\mbox{ for } & 0\leq x\leq1/2,\\
(1-x)/x & \,\,\mbox{ for } & 1/2\leq x\leq1.\end{array}\right.\]
 One immediately verifies that the inverse branches of $F$ are given
by $f_{1}(x)=x/(x+1)$ and $f_{2}(x)=1/(x+1)$. Also, let $\nu_{F}$
refer to the measure of maximal entropy of the system $(\mathcal{U},F)$.
That is, in particular, we have that $\nu_{F}(T_{n,k})=2^{-n}$, for
all $n\in\N_{0}$ and $k=0,\ldots,2^{n}-1$. Finally, note that $\nu_{F}$
is an $F$-invariant Gibbs measure for the  potential function equal to some constant.

\emph{The tent-system} $(\mathcal{U},T,\nu_{T})$: Let $T:\left[0,1\right]\rightarrow\left[0,1\right]$
refer to the \emph{tent map} on $\mathcal{U}$, given by \[
T\left(x\right):=\left\{ \begin{array}{lll}
2x & \,\,\mbox{ for } & 0\leq x\leq1/2,\\
2-2x & \,\,\mbox{ for } & 1/2<x\leq1.\end{array}\right.\]
 The measure of maximal entropy of the system $(\mathcal{U},T)$ will
be denoted by $\nu_{T}$, and we clearly have that $\nu_{T}=\lambda$.

\vspace{2mm} The following proposition shows that $\left(\mathcal{U},T\right)$
and $\left(\mathcal{U},F\right)$ are in fact topologically conjugate,
and that the conjugating homeomorphism is given by the distribution
function $\Delta_{\nu_{F}}$ of the Farey measure $\nu_{F}$. Moreover,
we will see that $\Delta_{\nu_{F}}$ is in fact equal to
$Q$.
Recall that Denjoy \cite{Denjoy0} \cite{Denjoy} and Salem \cite{Salem}
showed that $Q$ is given by \begin{equation}
Q(x)=-2\sum_{k\in\N}(-1)^{k}\,\,2^{-\sum_{i=1}^{k}a_{i}},\,\,\textrm{for all }\,\, 
x=[a_{1},a_{2},\ldots]\in\mathcal{U}.\label{alter}\end{equation}
 The following commuting diagram summarises the statement of the proposition.
\[
\begin{CD}\left(\mathcal{U},\nu_{F}\right)@>F>>\left(\mathcal{U},\nu_{F}\right)\\
@V\Delta_{\nu_{F}}=QVV@VV\Delta_{\nu_{F}}=QV\\
\left(\mathcal{U},\nu_{T}\right)@>T>>\left(\mathcal{U},\nu_{T}\right)\end{CD}\]
 Let us also remark that we believe that the proposition is well known
to experts in this area. However, we were unable to locate it
in the literature, and therefore decided to include the proof.

\begin{prop}
\label{lem:The-distribution-function} The two systems $\left(\mathcal{U},T\right)$
and $\left(\mathcal{U},F\right)$ are topologically conjugate, and
the conjugating homeomorphism is given by the distribution function
$\Delta_{\nu_{F}}$ of the Farey-measure $\nu_{F}$. Moreover, the
function $\Delta_{\nu_{F}}$ coincides with the Minkowski question
mark function $Q$. \end{prop}
\begin{proof}
Let us first show that $\Delta_{\nu_{F}}$ and $Q$ do in fact coincide.
For this, let $x=[a_{1},a_{2},\ldots]\in\mathcal{U}$ be given. Recall
that for the sequence $\left(p_{k}/q_{k}\right)_{k\in\N}$ of convergents
of $x$ (the sequence is finite if $x$ is rational, and infinite
otherwise)  we have that $p_{k}/q_{k}=[a_{1},\ldots,a_{k}]$, and
that $x=\lim_{k\rightarrow\infty}p_{k}/q_{k}$. Clearly, the latter
fact guarantees that it is sufficient to show that $\Delta_{\nu_{F}}(p_{k}/q_{k})=Q(p_{k}/q_{k})$,
for each of the convergents of $x$. For this, we employ the following
straightforward inductive argument. For ease of exposition, let $Q_{k}:=\Delta_{\nu_{F}}(p_{k}/q_{k})$
and $A_{k}:=|Q_{k+1}-Q_{k-1}|$. For the start of the induction, note
that if $a_{1}=1$ then $\Delta_{\nu_{F}}(1/a_{1})=1=Q(1)$. Similarly,
for $a_{1}>1$ we have \[
\Delta_{\nu_{F}}(1/a_{1})=1-\sum_{i=1}^{a_{1}-1}2^{-i}=
1-(1-2^{-(a_{1}-1)})=2\cdot 2^{-a_{1}}
=Q(1/a_{1}).\]
 For the inductive step, let us first state the following relations
(which will be verified in what follows). For each $k\in\N\,(k\neq1)$
we have \begin{equation}
A_{k+1}=(1-2^{-a_{k+1}})\,|Q_{k}-Q_{k-1}|,\,\,\textrm{and }\,\, Q_{k+1}=\left\{ \begin{array}{lll}
Q_{k-1}+A_{k+1} & \,\,\mbox{ for $k$ odd }\\
Q_{k-1}-A_{k+1} & \,\,\mbox{ for $k$ even. }\end{array}\right.\label{Q}\end{equation}
 The inductive assumption then is that $Q_{i}=Q(p_{i}/q_{i})$ holds
for each each $i\in\{1,\ldots,k\}$, for some $k\in\N$. Using this
and (\ref{Q}), it follows for $k$ odd, \begin{eqnarray*}
\Delta_{\nu_{F}}(p_{k+1}/q_{k+1}) & = & Q_{k+1}=Q_{k-1}+A_{k+1}=Q_{k-1}+(1-2^{-a_{k+1}})\,
|Q_{k}-Q_{k-1}|\\
 & = & Q_{k-1}+2\cdot2^{-\sum_{i=1}^{k}a_{i}}(1-2^{-a_{k+1}})=Q_{k-1}+2\cdot2^{-\sum_{i=1}^{k}a_{i}}-2\cdot2^{-\sum_{i=1}^{k+1}a_{i}}\\
 & = & -2\sum_{m=1}^{k+1}(-1)^{m}\,\,2^{-\sum_{i=1}^{m}a_{i}}=Q(p_{k+1}/q_{k+1}).
 \end{eqnarray*}
 Clearly, for $k$ even one can argue almost in the same way, and
this is left to the reader. This completes the inductive argument.\\
 We now still have to prove the assertions in (\ref{Q}). We do this 
 only
 for the case $k$ even, and leave `$k$ odd' up to the reader.
Recall that the interval bounded by $p_{k-1}/q_{k-1}$ and $p_{k+1}/q_{k+1}$
can be partitioned by the intermediate convergents $p_{k,m}/q_{k,m}$
of $x$. Here, $p_{k,m}/q_{k,m}$ is given by (see e.g. \cite{Khintchine},
see also Fig. \ref{cap:Intermediate-and-micro-intermediate} in the
proof of Proposition \ref{proposition}) \[
p_{k,m}:=mp_{k}+p_{k-1}\,\,\textrm{and }\,\,
q_{k,m}:=mq_{k}+q_{k-1},\,\,\textrm{for all }\,\, m \in \{0,\ldots,a_{k+1}\}.\]
 Then note that since $p_{k,1}/q_{k,1}$ is the mediant of $p_{k}/q_{k}$
and $p_{k-1}/q_{k-1}$, it follows that $|Q(p_{k,1}/q_{k,1})-Q(p_{k}/q_{k})|=2^{-1}|Q_{k}-Q_{k-1}|$.
Likewise, $p_{k,2}/q_{k,2}$ is the mediant of $p_{k,1}/q_{k,1}$
and $p_{k-1}/q_{k-1}$, and hence $|Q(p_{k,2}/q_{k,1})-Q(p_{k,1}/q_{k,1})|=
2^{-1}|Q_{k-1}-Q(p_{k,1}/q_{k,1})|=2^{-2}|Q_{k}-Q_{k-1}|$.
Clearly, this process can be continued until it terminates after $a_{k+1}$
steps. In the final step we obtain the identity $|Q_{k+2}-Q(p_{k,a_{k+1}-1}/q_{k,a_{k+1}-1})|=
2^{-a_{k+1}}|Q_{k}-Q_{k-1}|$.
The summation of these steps then gives \[
A_{k+1}=|Q_{k+2}-Q_{k}|=|Q_{k}-Q_{k-1}|\sum_{i=1}^{a_{k+1}}2^{-i}=(1-2^{-a_{k+1}})\,
|Q_{k}-Q_{k-1}|.\]
 This proves the first assertion in (\ref{Q}). The second assertion
in (\ref{Q}) is an immediate consequence of the well known fact that
the value of $x$ is greater than any of its even-order convergents
and is less than any of its odd-order convergents (see e.g. 
\cite{Khintchine}). This finishes the
proof of the equality of $\Delta_{\nu_{F}}$ and $Q$.\\
 For the proof of $T\circ\Delta_{\nu_{F}}=\Delta_{\nu_{F}}\circ F$,
note that if $x=\left[a_{1},a_{2},\ldots\right]$ is such that $a_{1}>1$,
then (\ref{alter}) gives \begin{eqnarray*}
T\left(\Delta_{\nu_{F}}\left(x\right)\right) & = & T\left(Q\left(x\right)\right)=
2\left(-2\sum_{k\in\N}(-1)^{k}\,\,
2^{-\sum_{i=1}^{k}a_{i}}\right)=-2\sum_{k\in\N}(-1)^{k}\,\,2^{-\sum_{i=1}^{k}a_{i}-1}\\
 & = & Q\left(\left[a_{1}-1,a_{2},\ldots\right]\right)=Q\left(x/(1-x)\right)=
 Q\left(F\left(x\right)\right)=\Delta_{\nu_{F}}\left(F\left(x\right)\right).\end{eqnarray*}
 Similar, for $x=\left[1,a_{2},\ldots\right]$ we have\begin{eqnarray*}
T\left(\Delta_{\nu_{F}}\left(x\right)\right) & = & T\left(Q\left(x\right)\right)=
2-2\left(-2\sum_{k\in\N}(-1)^{k}\,\,
2^{-\sum_{i=1}^{k}a_{i}}\right)=-2\sum_{k\in\N}(-1)^{k}\,\,2^{-\sum_{i=1}^{k}a_{i+1}}\\
 & = & Q\left((1-x)/x\right)=Q\left(F\left(x\right)\right)=
 \Delta_{\nu_{F}}\left(F\left(x\right)\right).\end{eqnarray*}
 Finally, the fact that $\Delta_{\nu_{F}}$ is a homeomorphism is
an immediate consequence of its construction. This finishes the proof.
\end{proof}
\begin{rem}
(1)~ An immediate implication of Proposition \ref{lem:The-distribution-function}
is that $\nu_{F}=\nu_{T}\circ Q$, and that the measure theoretical
and topological entropies $h_{\nu_{F}}(F)$, $h_{\nu_{T}}(T)$,  $h_{\mathrm{top}}\left(T\right)$
and $h_{\mathrm{top}}\left(F\right)$ of both systems coincide and
are equal to $h_{\mathrm{top}}:=\log2$. In fact, this also leads
to an alternative proof of the fact that $Q$ represents the distribution
function of $\nu_{F}$. Namely, \[
\nu_{F}\left([0,x)\right)=\nu_{T}\circ Q([0,x))=\lambda\circ Q([0,x))=
Q\left(x\right),\,\,\textrm{for each }\,\, x\in\mathcal{U}.\]
 (2) ~ Let us also remark that by the above, we immediately
have that \begin{equation}
Q\left(s_{n,k}/t_{n,k}\right)=k\,2^{-n},Q(T_{n,k})=D_{n,k},\,\,\textrm{and }\,\,
\nu_{F}(T_{n,k})=\lambda(Q(T_{n,k}))=2^{-n}.\label{form2}\end{equation}
 Also, the reader might like to recall that $Q$ is related to the
Stern-Brocot sequence $\left(s_{n,k}/t_{n,k}\right)$ in the following
way. We clearly have $Q(s_{0,0}/t_{0,0})=0$ and $Q(s_{0,1}/t_{0,1})=1$.
Moreover, for two neighbours in the $n$-th Stern-Brocot sequence,
we have \[
Q\left(\frac{s_{n,k}+s_{n,k+1}}{t_{n,k}+t_{n,k+1}}\right)=\frac{1}{2}
\left(Q\left(\frac{s_{n,k}}{t_{n,k}}\right)+Q\left(\frac{s_{n,k+1}}{t_{n,k+1}}\right)\right).\]
 Finally, recall that $x$ is rational if and only if $Q(x)$ has
a finite dyadic expansion, and that $x$ is a quadratic surd if and
only if $Q(x)$ is a rational number with an infinite dyadic expansion.
In fact, the latter two properties of $Q$ were Minkowski's original
main motivation for introducing the function $Q$ in the first place.
\end{rem}

\section{The integral of the Minkowski function w.r.t. the Gauss measure}

The following proposition gives the main result of this section. For
this recall that the Hausdorff dimension of a probability measure $\mu$ is given
by (see e.g. \cite{Fa}) \[
\dim_{H}(\mu):=\inf\left\{ \dim_{H}(X):\mu(X)=1\right\} .\]
 Also, let ${\mathbb E}_{\mu}(\Delta_{\nu}):=\int\Delta_{\nu}\, d\mu$
refer to the $\mu$-expectation of the distribution function $\Delta_{\nu}\in L^{1}(\mathcal{U},\mu)$
of $\nu$, for two probability measures $\nu$ and $\mu$ on $\mathcal{U}$.
Moreover, let $m_{G}$ refer to the Gauss measure. That is, $m_{G}$
refers to the invariant measure of the Gauss map $G:x\mapsto1/x\mod(1)$
absolutely continuous to $\lambda$.
\begin{prop}
\label{Lemma:Kinney} For the $m_{G}$-expectation of $\Delta_{\nu_{F}}$
and the $\nu_{F}$-expectation of $\Delta_{m_{G}}$, we have \[
{\mathbb E}_{m_{G}}(Q)=\frac{\dim_{H}(\nu_{F})-1/2}{\dim_{H}(\nu_{F})}\,\,\,\,
\textrm{and }\,\,\,\,{\mathbb E}_{\nu_{F}}(\Delta_{m_{G}})=\frac{1}{2\dim_{H}(\nu_{F})}.\]
\end{prop}
\begin{proof}
First note that the Stern-Brocot pressure function at zero corresponds
to the Legendre transform $\widehat{P}$ at $-\chi_{\nu_{F}}$, where
$\chi_{\nu_{F}}:=\int\log|F'|\, d\nu_{F}$ denotes the Lyapunov exponent
of $F$. That is, \[
\widehat{P}(-\chi_{\nu_{F}})=\sup_{t\in\R}\{-t\cdot\chi_{\nu_{F}}-P(t)\}=
-\,0\cdot\chi_{\nu_{F}}-P(0)=-h_{\mathrm{top}}.\]
 Combining this observations with the fact that $\nu_{F}$ is the
$F$-invariant Gibbs measure associated with $\mathcal{L}(\chi_{\nu_{F}})$,
Theorem \ref{KS} implies \begin{equation}
\dim_{H}(\nu_{F})=\dim_{H}\left(\mathcal{L}(\chi_{\nu_{F}})\right)=
-\hat{P}(-\chi_{\nu_{F}})/\chi_{\nu_{F}}=h_{\mathrm{top}}/\chi_{\nu_{F}}.\label{R}\end{equation}
 Hence we are left with to determine $\chi_{\nu_{F}}$ in terms of
${\mathbb E}_{\nu_{F}}(\Delta_{m_{G}})$. For this, recall that for
the distribution function $\Delta_{m_{G}}$ of $m_{G}$ we have \[
\Delta_{m_{G}}(x):=m_{G}\left([0,x)\right)=\int_{0}^{x}1/(1+x)\,
d\lambda(x)/h_{\mathrm{top}}=\log(1+x)/h_{\mathrm{top}},\,\,\hbox{for all}\,\, x\in\mathcal{U}.\]
 Combining this with a straightforward computation of $|F'|$, one
immediately verifies \[
\log|F'|=2\, h_{\mathrm{top}}\cdot\left(\Delta_{m_{G}}\circ F\right).\]
 Hence, using the $F$-invariance of $\nu_{F}$, it follows \begin{eqnarray*}
\chi_{\nu_{F}}=\int\log|F'|\, d\nu_{F}=2\, h_{\mathrm{top}}\int\Delta_{m_{G}}\circ F\,
d\nu_{F}=2\, h_{\mathrm{top}}\int\Delta_{m_{G}}\, d\nu_{F}=2\, h_{\mathrm{top}}\,
{\mathbb E}_{\nu_{F}}(\Delta_{m_{G}}).\end{eqnarray*}
 By inserting this into (\ref{R}) and solving for ${\mathbb E}_{\nu_{F}}(m_{G})$,
the second equality in the proposition follows. The first equality
in the proposition is now an immediate consequence of the fact that
\[
{\mathbb E}_{m_{G}}(\Delta_{\nu_{F}})=1-{\mathbb E}_{\nu_{F}}(\Delta_{m_{G}}).\]
 Since by Proposition \ref{lem:The-distribution-function} we have
$\Delta_{\nu_{F}}=Q$, this finishes the proof.
\end{proof}
As an immediate consequence of Proposition \ref{Lemma:Kinney} we
obtain the following result of Kinney \cite{Kinney}, which we state in 
its `non-dynamical'
form in which it was given in \cite{Kinney}.

\begin{cor}
There exists a set $A\subset\mathcal{U}$ such that $\lambda(Q(A))=1$,
and \[
\dim_{H}(A)=\left(2\int_{0}^{1}\log_{2}(1+x)\,\, dQ(x)\right)^{-1}.\]

\end{cor}
\begin{proof}
Note that for the derivative $\left(f_{i}\right)'$ of the inverse
branches of $F$ we have \[
\left(f_{i}\right)'(x)=(1+x)^{-2},\,\textrm{ for all }\, x\in\mathcal{U},\, i\in\{1,2\}.\]
 Using this and the $F$-invariance of $\nu_{F}$, it follows \begin{eqnarray*}
\chi_{\nu_{F}} & = & \int\log|F'|\, d\nu_{F}=\int\left(\1_{[0,1/2)}
\log|F'\circ f_{1}\circ F|+\1_{[1/2,1]}\log|F'\circ f_{2}\circ F|\right)\, d\nu_{F}\\
 & = & -\int\log|\left(F^{-1}\right)'\circ F|\, d\nu_{F}=-\int\log|\left(F^{-1}\right)'|
 \, d\nu_{F}\\
 & = & \int_{\mathcal{U}}\log\left((1+x)^{2}\right)\, d\nu_{F}(x).\end{eqnarray*}
 Inserting this into (\ref{R}), the result follows.
\end{proof}
\begin{rem}
\label{remarkK} Note that in \cite{TU} the numerical
approximation $\dim_{H}(\nu_{F})\approx7/8$ was obtained (see also
\cite{L}). Hence, for the Stern-Brocot rate $\chi_{\nu_{F}}$
associated with $\nu_{F}$ we have that $\chi_{\nu_{F}}=h_{\mathrm{top}}/\dim_{H}(\nu_{F})\approx0.792$,
or in other words, $\ell(x)\approx0.792$ for $\nu_{F}$-almost every
$x\in\mathcal{U}$. Moreover, this also immediately gives ${\mathbb E}_{m_{G}}(\Delta_{\nu_{F}})\approx3/7$
and ${\mathbb E}_{\nu_{F}}(\Delta_{m_{G}})\approx4/7$. (In fact,
for the latter we derived, using numerical integration, the slightly better approximation
${\mathbb E}_{\nu_{F}}(\Delta_{m_{G}})=0.571612\ldots$). 

\end{rem}
Let us end this section by showing that the H\"older continuity of
$Q$ reflects precisely the range $[0,2\log\gamma]$ of the Lyapunov
spectrum associated with $\ell$. For this, note that Salem showed
in \cite{Salem} that $Q$ is $(\log2/(2\log\gamma))$-H\"older
continuous. That is, \[
|Q(x)-Q(y)|\ll\,|x-y|^{\log2/(2\log\gamma)},\,\textrm{ for all }\, x,y\in\mathcal{U}.\]
 (Note that $\log2/(2\log\gamma)\approx0.7202$). As a consequence
of this modulus of continuity of $Q$ we have the following.

\begin{lem}
For each $x\in\mathcal{U}$, we have\[
\limsup_{n\in\N}\,\ell_{n}(x)\leq2\log\gamma.\]
 Here, the constant $2\log\gamma\approx0.9624$ is best possible,
since it is attained for instance for each \emph{noble number}, that
is a number whose continued fraction expansion eventually contains
only $1$'s, and hence it is attained in particular for $x=\gamma^{*}:=1/\gamma$.
\end{lem}
\begin{proof}
The $(\log2/(2\log\gamma))$-H\"older continuity of $Q$ implies
that for each $x\in\mathcal{U}$ and $n\in\N$, we have \[
\nu_{F}(T_{n}(x))=\lambda\left(Q(T_{n}(x))\right)\ll\left(\lambda\left(T_{n}(x)\right)\right)^{\log2/(2\log\gamma)}.\]
 This implies, with $C>0$ referring to some universal constant,
\[
-n\log2=\log\nu_{F}(T_{n}(x))\leq\frac{\log2}{2\log\gamma}\log\lambda\left(T_{n}(x)\right)+C,\]
 which gives \[
\limsup_{n\in\N}\,\ell_{n}(x)\leq2\log\gamma.\]
 For the remaining assertion recall that numerator and denominator
of the $n$-th convergent $p_{n}/q_{n}:=p_{n}(\gamma^{*})/q_{n}(\gamma^{*})$
of $\gamma^{*}$ are equal to the $n$-th and $(n+1)$-th member
of the Fibonacci sequence. That is, \[
p_{n}=\left(\gamma^{n}-\left(-\gamma^{*}\right)^{n}\right)/\sqrt{5}\,\,\,\textrm{and }\,\,\, q_{n}=p_{n+1}.\]
 Using this together with a well known Diophantine identity for continued
fractions (see e.g. \cite{Khintchine}), one immediately obtains,
with $\left(O_{i,n}\right)$ referring to certain sequences which
tend to zero for $n$ tending to infinity, 
\begin{eqnarray*}
\left|\gamma^{*}-p_{n}/q_{n}\right| & = & \frac{1}{q_{n}^{2}(\gamma+p_{n}/q_{n})}=\frac{1}{q_{n}^{2}(\sqrt{5}+O_{1,n})}
=\frac{5}{\sqrt{5}+O_{1,n}}\left(\gamma^{n+1}-(-\gamma^{*})^{n+1}\right)^{-2}\\
 & = & \frac{\sqrt{5}+O_{2,n}}{\gamma^{2}+O_{3,n}}\gamma^{-2n}
=\left(\gamma^{-2}\sqrt{5}+O_{4,n}\right)\gamma^{-2n}.
\end{eqnarray*}
Note that $Q(\gamma^{*})=\sum_{i=0}^{\infty}(-2)^{-i}$ and
$Q(p_{n}/q_{n})=\sum_{i=0}^{n-1}(-2)^{-i}$, and hence, with $O_{5,n}:=|\sum_{i=n+1}^{\infty}(-2)^{-i}|$,
\[
\left|Q(\gamma^{*})-Q(p_{n}/q_{n})\right|=2^{-n}-O_{5,n}.\]
 Combining these two observations, it follows 
\begin{eqnarray*}
\left|Q(\gamma^{*})-Q(p_{n}/q_{n})\right| & = & \left(\gamma^{-2n}\right)^{\log2/(2\log\gamma)}-O_{5,n}\\
 & = & \left(\gamma^{-2}\sqrt{5}+O_{4,n}\right)^{-\log2/(2\log\gamma)}
\left|\gamma^{*}-p_{n}/q_{n}\right|^{\log2/(2\log\gamma)}-O_{5,n}.
\end{eqnarray*}
 By taking logarithms, the result follows.
\end{proof}

\section{The derivative of the Minkowski function}

Let us begin our analysis of the derivative of $Q$ with the following
lemma. Note that the instance in which either $Q'(x)$ exists or $Q'(x)=\infty$
will be referred to as $Q'(x)$ \emph{exists in the generalised sense}.

\begin{lem}
\label{Lemma0} For each $x\in\mathcal{U}$ we have that if $Q'(x)$
exists in the generalised sense, then \[
Q'(x)=\lim_{n\rightarrow\infty}\frac{\nu_{F}(T_{n}(x))}{\lambda(T_{n}(x))}.\]

\end{lem}
\begin{proof}
Let $x\in\mathcal{U}$ be given, and assume that $Q'(x)$ exists in
the generalised sense. Let $T_{n}(x)=[s_{n,k}/t_{n,k},s_{n,k+1}/t_{n,k+1})$
be the unique Stern-Brocot interval in $\mathcal{T}_{n}$ which contains
$x$. Note that the alternating sum representation (\ref{alter})
of $Q$ immediately gives that $Q$ is a strictly increasing function.
Using this, it follows that for each $n\in\N$ one of the following
two cases has to occur. Firstly, if $Q(x)$ lies below or on the line
through $Q(s_{n,k}/t_{n,k})$ and $Q(s_{n,k+1}/t_{n,k+1})$, then
\[
\frac{Q(x)-Q(s_{n,k}/t_{n,k})}{x-s_{n,k}/t_{n,k}}\leq\frac{Q(s_{n,k+1}/t_{n,k+1})-Q(s_{n,k}/t_{n,k})}{s_{n,k+1}/t_{n,k+1}-s_{n,k}/t_{n,k}}\leq\frac{Q(s_{n,k+1}/t_{n,k+1})-Q(x)}{s_{n,k+1}/t_{n,k+1}-x}.\]
 Secondly, if $Q(x)$ lies above or on the line through $Q(s_{n,k}/t_{n,k})$
and $Q(s_{n,k+1}/t_{n,k+1})$, then \[
\frac{Q(s_{n,k+1}/t_{n,k+1})-Q(x)}{s_{n,k+1}/t_{n,k+1}-x}\leq\frac{Q(s_{n,k+1}/t_{n,k+1})-Q(s_{n,k}/t_{n,k})}{s_{n,k+1}/t_{n,k+1}-s_{n,k}/t_{n,k}}\leq\frac{Q(x)-Q(s_{n,k}/t_{n,k})}{x-s_{n,k}/t_{n,k}}.\]
 Hence, by taking the limit for $n$ tending to infinity, and noting
that \[
Q(s_{n,k+1}/t_{n,k+1})-Q(s_{n,k}/t_{n,k})=\nu_{F}([0,s_{n,k+1}/t_{n,k+1}))-\nu_{F}([0,s_{n,k}/t_{n,k}))=\nu_{F}(T_{n}(x)),\]
 the assertion follows.
\end{proof}
The following result was obtained in \cite{PVB:01} using continued
fraction expansions. Here, we give an alternative proof which uses
Stern-Brocot sequences, and which  appears to us to be far more canonical
than the one given in \cite{PVB:01}.

\begin{lem}
\label{Lemma1} For each $x\in\mathcal{U}$ we have that if $Q'(x)$
exists in the generalised sense, then \[
Q'(x)\in\{0,\infty\}.\]

\end{lem}
\begin{proof}
Let $x\in\mathcal{U}$ be given such that $Q'(x)$ exists in the generalised
sense. Without loss of generality we can assume that $x$ is irrational.
By Lemma \ref{Lemma0}, we then have \[
Q'(x)=\lim_{n\rightarrow\infty}\frac{\nu_{F}(T_{n}(x))}{\lambda(T_{n}(x))}.\]
 Let us assume by way of contradiction that $Q'(x)=c$, for some $0<c<\infty$.
Since we have $Q'(x)=\lim_{n\rightarrow\infty}2^{-n}/\lambda\left(T_{n}(x)\right)$,
it follows that  \[
\lim_{n\rightarrow\infty}\frac{2^{n}\,\lambda\left(T_{n}(x)\right)}{2^{n+1}\,\lambda\left(T_{n+1}(x)\right)}=1,\]
 and hence, \begin{equation}
\lim_{n\rightarrow\infty}\frac{\lambda\left(T_{n}(x)\right)}{\lambda\left(T_{n+1}(x)\right)}=2.\label{form3}\end{equation}
 In order to proceed, let $T_{n}(x)=[s_{n,k}/t_{n,k},s_{n,k+1}/t_{n,k+1})$,
and assume that there is a `type-change' at $T_{n}(x)$. That is,
assume that $T_{n-1}(x)=[(s_{n,k}-s_{n,k+1})/(t_{n,k}-t_{n,k+1}),s_{n,k+1}/t_{n,k+1})$
and $T_{n+1}(x)=[s_{n,k}/t_{n,k},(s_{n,k}+s_{n,k+1})/(t_{n,k}+t_{n,k+1}))$.
We then immediately obtain \[
\frac{\lambda\left(T_{n}(x)\right)}{\lambda\left(T_{n+1}(x)\right)}=\frac{s_{n,k+1}/t_{n,k+1}-s_{n,k}/t_{n,k}}{(s_{n,k}+s_{n,k+1})/(t_{n,k}+t_{n,k+1})-s_{n,k}/t_{n,k}}=\frac{t_{n,k}(t_{n,k}+t_{n,k+1})}{t_{n,k}t_{n,k+1}}=1+\frac{t_{n,k}}{t_{n,k+1}}.\]
 Combining this with (\ref{form3}), it follows \begin{equation}
\lim_{n\rightarrow\infty}\frac{t_{n,k}}{t_{n,k+1}}=1.\label{formA}\end{equation}
 By considering the quotient of $\lambda\left(T_{n-1}(x)\right)$
and $\lambda\left(T_{n}\left(x\right)\right)$, a similar computation
gives \begin{equation}
\frac{\lambda\left(T_{n-1}(x)\right)}{\lambda\left(T_{n}(x)\right)}=\frac{t_{n,k}}{t_{n,k}-t_{n,k+1}}=\frac{1}{1-t_{n,k+1}/t_{n,k}}.\label{formB}\end{equation}
 Then observe that since $x$ is irrational, there have to be infinitely
many type-changes in $\{ T_{n}(x):n\in\N\}$. That is, there exist
sequences $(n_{i})_{i\in\N}$ and $(k_{i})_{i\in\N}$ such that $T_{n_{i}}(x)=[s_{n_{i},k_{i}}/t_{n_{i},k_{i}},s_{n_{i},k_{i}+1}/t_{n_{i},k_{i}+1})$,
and such that there is a type-change at $T_{n_{i}}(x)$ for each
$i\in\N$. Therefore, combining this with (\ref{formA}) and (\ref{formB}),
it now follows \[
\lim_{i\rightarrow\infty}\frac{\lambda(T_{n_{i}-1}(x))}{\lambda\left(T_{n_{i}}(x)\right)}=\lim_{i\rightarrow\infty}\frac{1}{1-t_{n_{i},k_{i}+1}/t_{n_{i},k_{i}}}=\infty.\]
 This contradicts (\ref{form3}), and hence finishes the proof of
the lemma.
\end{proof}

The following proposition will turn out to be crucial in the 
multifractal analysis to come. For ease of exposition, we let 
$[x,y)_{\pm}$
refer to the interval bounded by $x$ and $y$. That is, $[x,y)_{\pm}:=
[x,y)$ if $x\leq y$, and $[x,y)_{\pm}:= [y,x)$
if $x \geq y$.
\begin{prop}
\label{proposition} For $x=[a_{1},a_{2},\ldots]\in\mathcal{U}$ and with
 $p_{k}/q_{k}$ referring to the $k$-th convergent of $x$, the 
 following hold.
\begin{itemize}
\item[(i)] 
\[
\hbox{If} \, \, \,\lim_{k\rightarrow\infty}\frac{\nu_{F}
\left([p_{k}/q_{k},p_{k+1}/
q_{k+1})_{\pm}\right)}{\lambda\left([p_{k}/q_{k},p_{k+1}/q_{k+1})_{\pm}\right)}
=\infty, \, \, \hbox{then} \, \,\,
Q'(x)=\infty.\]
\item[(ii)] \[ \hbox{If}\, \, \, 
\lim_{k\rightarrow\infty} a_{k+1} \cdot   \frac{\nu_{F} \left([p_{k}/q_{k},p_{k+1}/
q_{k+1})_{\pm}\right)}{\lambda\left([p_{k}/q_{k},p_{k+1}/
q_{k+1})_{\pm}\right)}
=0, \, \, \hbox{then} \, \, \,
Q'(x)=0.\]
\end{itemize}
\end{prop}
\begin{proof}
Let $x=[a_{1},a_{2},\ldots]\in\mathcal{U}$ be given as stated in
(i). Using  (\ref{alter}) and the fact that $|p_{k}q_{k+1}-p_{k+1}q_{k}|=1$,
we immediately obtain \begin{eqnarray}
\frac{\nu_{F}\left([p_{k}/q_{k},p_{k+1}/q_{k+1})_{\pm}
\right)}{\lambda\left([p_{k}/q_{k},p_{k+1}/q_{k+1})_{\pm}\right)}=\frac{|Q(p_{k}/q_{k})-Q(p_{k+1}/q_{k+1})|}{|p_{k}/q_{k}-p_{k+1}/q_{k+1}|}=\frac{2q_{k}q_{k+1}}{2^{\sum_{i=1}^{k+1}a_{i}}}.\label{B}\end{eqnarray}
 Before we proceed, let us first recall that the intermediate convergents
$p_{k,m}/q_{k,m}$ of $x$ are given by (see e.g. \cite{Khintchine})
\[
p_{k,m}:=mp_{k}+p_{k-1}\,\,\textrm{and }\,\, q_{k,m}:=mq_{k}+q_{k-1},\,\,\textrm{for all }\,\,
m\in \{0,\ldots,a_{k+1}\}.\]
 Since $p_{k,0}/q_{k,0}=p_{k-1}/q_{k-1}=[a_{1},\ldots,a_{k-1}],p_{k,a_{k+1}}/q_{k,a_{k+1}}=p_{k+1}/q_{k+1}$
and $p_{k,n}/q_{k,n}=[a_{1},\ldots,a_{k},n]$ for $n\in \{1,\ldots,a_{k+1}\}$,
we immediately obtain from (\ref{alter}) that for each $m\in \{0,\ldots,a_{k+1}-1\}$,
\[
|Q(x)-Q(p_{k,m}/q_{k,m})|\gg2^{-(m+\sum_{j=1}^{k}a_{j})},\]
 and \[
|Q(x)-Q(p_{k,a_{k+1}}/q_{k,a_{k+1}})|\gg2^{-\sum_{j=1}^{k+2}a_{j}}.\]
 We then compute for $m\in \{0,\ldots,a_{k+1}\}$, with $r_{n}:=[a_{n};a_{n+1},\ldots]$
referring to the $n$-th remainder of $x$, \begin{eqnarray*}
|x-p_{k,m}/q_{k,m}| & = & \left|\frac{r_{k+1}p_{k}+p_{k-1}}{r_{k+1}q_{k}+q_{k-1}}-\frac{mp_{k}+p_{k-1}}{mq_{k}+q_{k-1}}\right|=\frac{r_{k+1}-m}{(r_{k+1}q_{k}+q_{k-1})(mq_{k}+q_{k-1})}\\
\end{eqnarray*}
 Now, let $y\in\mathcal{U}$ be fixed such that $y>x$. Then there
exist $k\in\N$ and $m\in\{0,\ldots,a_{k+1}-1\}$ such that $p_{k,m+1}/q_{k,m+1}<y\leq p_{k,m}/q_{k,m}$.
For each $m\in \{0,\ldots,a_{k+1}-2\}$, we then have \begin{eqnarray*}
\frac{Q(y)-Q(x)}{y-x} & \geq & \frac{Q(p_{k,m+1}/q_{k,m+1})-Q(x)}{p_{k,m}/q_{k,m}-x}\\
 & \gg & \frac{(r_{k+1}q_{k}+q_{k-1})(mq_{k}+q_{k-1})}{q_{k}q_{k+1}(r_{k+1}-m)}\frac{q_{k}q_{k+1}}{2^{(m+1)+\sum_{j=1}^{k}a_{j}}}\\
 & = & \frac{2^{a_{k+1}-(m+1)}(mq_{k}+q_{k-1})(r_{k+1}q_{k}+q_{k-1})}{(r_{k+1}-m)q_{k}q_{k+1}}\frac{q_{k}q_{k+1}}{2^{\sum_{j=1}^{k+1}a_{j}}}\\
 & \gg & \frac{q_{k}q_{k+1}}{2^{\sum_{j=1}^{k+1}a_{j}}}.\end{eqnarray*}
 Note that the latter argument does not work for $m=a_{k+1}-1$.
In this case, that is for $p_{k+1}/q_{k+1}<y\leq p_{k,a_{k+1}-1}/q_{k,a_{k+1}-1}$,
we have to consider the partition of the interval $(p_{k+1}/q_{k+1},p_{k,a_{k+1}-1}/q_{k,a_{k+1}-1}]$
obtained from what we call the `micro-intermediate
convergents' $\hat{p}_{k,n}/\hat{q}_{k,n}$
\begin{figure}[b]
\psfrag{x2}{\Large\({\frac{p_{k-1}}{q_{k-1}}}\)}

 \psfrag{x1}{\Large\(\frac{p_k}{q_k}\)}

\psfrag{x}{\LARGE\(x\)}

\psfrag{x3}{ \Large\(\frac{p_{k,1}}{q_{k,1}}\)}

\psfrag{x4}{ \Large\(\frac{p_{k,a_{k+1}-1}}{q_{k,a_{k+1}-1}}\)}

\psfrag{x5}{\Large\(\frac{p_{k+1}}{q_{k+1}}\)}

\psfrag{x6}{ \Large\(\frac{\widehat{p}_{k,2}}{\widehat{q}_{k,2}}\)}

\psfrag{x7}{ \Large\(\frac{\widehat{p}_{k,3}}{\widehat{q}_{k,3}}\)}
\psfrag{a}{micro-intermediate}
\psfrag{b}{intermediate}\includegraphics{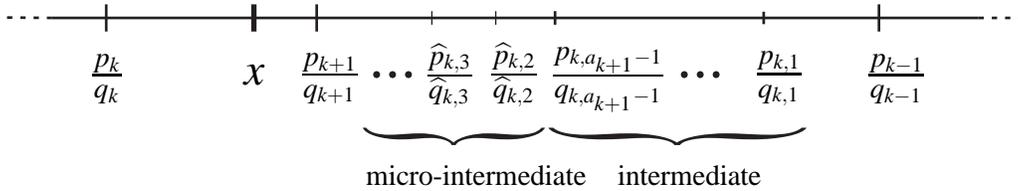}

\caption{\label{cap:Intermediate-and-micro-intermediate} Regular, intermediate,
and micro-intermediate convergents for $x\in\mathcal{U}$ and $k\in\N$
even.}
\end{figure}
(cf. Fig. \ref{cap:Intermediate-and-micro-intermediate}). These are
given for $n\in\N$ by \[
\hat{p}_{k,n}:=np_{k+1}-p_{k}\,\,\textrm{and }\,\,\hat{q}_{k,n}:=nq_{k+1}-q_{k}.\]
 Note that $\hat{p}_{k,1}/\hat{q}_{k,1}=((a_{k+1}-1)p_{k}+p_{k-1})/((a_{k+1}-1)q_{k}+q_{k-1})=p_{k,a_{k+1}-1}/q_{k,a_{k+1}-1}$.
Also, one immediately verifies that the continued fraction expansion
of $\hat{p}_{k,n}/\hat{q}_{k,n}$ is given by \[
\hat{p}_{k,n}/\hat{q}_{k,n}=[a_{1},\ldots,a_{k},a_{k+1}-1,1,n].\]
 Clearly, if $y\in(p_{k+1}/q_{k+1},p_{k,a_{k+1}-1}/q_{k,a_{k+1}-1}]$,
then there exists $l\in\N$ such that $\hat{p}_{k,l+1}/\hat{q}_{k,l+1}<y\leq\hat{p}_{k,l}/\hat{q}_{k,l}$.
Using (\ref{alter}) together with the fact that $Q$ is strictly
increasing, one then immediately obtains the estimate \begin{eqnarray*}
Q(y)-Q(x) & \geq & Q(\hat{p}_{k,l+1}/\hat{q}_{k,l+1})-Q(x)\gg2^{-\sum_{i=1}^{k+1}a_{i}}\left(1-2^{-(l+1)}-2^{-a_{k+2}}\right)\\
 & \geq & 2^{-\sum_{i=1}^{k+1}a_{i}}.\end{eqnarray*}
 Furthermore, in this situation we trivially have \[
y-x\leq p_{k,a_{k+1}-1}/q_{k,a_{k+1}-1}-p_{k}/q_{k}\ll1/(q_{k}q_{k+1}).\]
 Hence, this shows that also in this case we have \begin{eqnarray*}
\frac{Q(y)-Q(x)}{y-x} & \gg & \frac{q_{k}q_{k+1}}{2^{\sum_{j=1}^{k+1}a_{j}}}.\end{eqnarray*}
 Combining the above with (\ref{B}) and the assumption in (i),
it now follows \[
\lim_{y\rightarrow x+}\frac{|Q(x)-Q(y)|}{|x-y|}\gg\lim_{k\rightarrow\infty}
\frac{q_{2k}q_{2k+1}}{2^{\sum_{j=1}^{2k+1}a_{j}}}=\lim_{k\rightarrow\infty}
\frac{\nu_{F}\left([p_{2k}/q_{2k},p_{2k+1}/q_{2k+1})\right)}{\lambda\left([p_{2k}/q_{2k},p_{2k+1}/q_{2k+1})\right)}=\infty.\]
 Clearly, a minor modification of the argument above then also gives
that for the limit from the left we have \[
\lim_{y\rightarrow x-}\frac{|Q(x)-Q(y)|}{|x-y|}\gg\lim_{k\rightarrow\infty}
\frac{q_{2k-1}q_{2}}{2^{\sum_{j=1}^{2k}a_{j}}}=\lim_{k\rightarrow\infty}
\frac{\nu_{F}\left([p_{2k-1}/q_{2k-1},p_{2k}/q_{2k})\right)}{\lambda\left([p_{2k-1}/q_{2k-1},p_{2k}/q_{2k})\right)}=\infty.\]
 Hence, we conclude that $Q'(x)=\infty$, and this finishes the proof
of the assertion in (i).\\
For the proof of (ii), we proceed similar as for (i). 
Namely, let $y\in\mathcal{U}$ be fixed such that $y>x$. Then
there
exist $k\in\N$ and $m\in\{0,\ldots,a_{k+1}-1\}$ such that $p_{k,m+1}/q_{k,m+1}<y\leq p_{k,m}/q_{k,m}$.
For each $m\in \{0,\ldots,a_{k+1}-2\}$, we then have \begin{eqnarray*}
\frac{Q(y)-Q(x)}{y-x} & \ll & \frac{Q(p_{k,m}/q_{k,m})-
Q(x)}{p_{k,m+1}/q_{k,m+1}-x}\\
 & \ll & \frac{(r_{k+1}q_{k}+q_{k-1})((m+1)q_{k}+q_{k-1})
 }{(r_{k+1}-(m+1)) q_{k-1}q_{k}} \, \, \frac{q_{k-1}q_{k}}{2^{m+
 \sum_{j=1}^{k}a_{j}}}\\
 & \ll & \frac{a_{k+1} (m+1) q_{k}^{2}}{2^{m}(a_{k+1}-(m+1))
 q_{k-1} q_{k}} \, \, 
 \frac{q_{k-1}q_{k}}{2^{\sum_{j=1}^{k}a_{j}}}\\
 & \ll & \frac{a_{k+1} (m+1)q_{k}}{2^{m}
 (a_{k+1}-(m+1))}
 \cdot a_{k}  \cdot \frac{q_{k-1}q_{k}}{2^{\sum_{j=1}^{k}a_{j}}}\\
 & \ll & a_{k} \cdot 
 \frac{\nu_{F} \left([p_{k-1}/q_{k-1},p_{k}/
 q_{k})_{\pm}\right)}{\lambda\left([p_{k-1}/q_{k-1},p_{k}/
 q_{k})_{\pm}\right)}.\end{eqnarray*}
For the remaining case  $m=a_{k+1}-1$, we observe
\[ (Q(y)-Q(x))/(y-x) \leq  Q(p_{k,a_{k+1}-1}/q_{k,a_{k+1}-1})-
Q(x) \ll 2^{-\sum_{j=1}^{k+1}a_{j}},\]
and 
\[ y-x \geq 1/(2 q_{k}q_{k+1}).\]
Therefore, also in this case we have
\[ \frac{Q(y)-Q(x)}{y-x}  \ll \frac{q_{k}q_{k+1}}{2^{\sum_{j=1}^{k+1}a_{j}}}
\leq a_{k+1} \cdot   \frac{\nu_{F} \left([p_{k}/q_{k},p_{k+1}/
q_{k+1})_{\pm}\right)}{\lambda\left([p_{k}/q_{k},p_{k+1}/
q_{k+1})_{\pm}\right)}.\]
A similar estimate can be given for $y<x$, and this is left to the 
reader. Clearly, using the assumption in (ii), we can now proceed 
as in the proof of (i), and  this then gives $Q'(x)=0$. This completes 
the proof of the proposition.
\end{proof}

\begin{rem}
\label{rem5.5} Note that the proof of Proposition \ref{proposition}
also shows that the following implication holds. \[
\limsup_{k\rightarrow\infty}\frac{\nu_{F}
\left([p_{k}/q_{k},p_{k+1}/q_{k+1})\right)}{\lambda
\left([p_{k}/q_{k},p_{k+1}/q_{k+1})\right)}=\infty\implies
\limsup_{y\rightarrow x}\frac{Q(x)-Q(y)}{x-y}=\infty.\]
Moreover, note that since \[
\left|Q(x)-Q(p_{k}/q_{k})\right|\leq
2^{1-\sum_{i=1}^{k+1}a_{i}}\mbox{ and }
\left|x-p_{k}/q_{k}\right|\geq\frac{1}{2q_{k}q_{k+1}}, \, \, \hbox{for 
all} \, \, k \in \N,\]
 we also have the  implication:\[
\liminf_{k\rightarrow\infty}\frac{\nu_{F}\left([p_{k}/q_{k},p_{k+1}/
q_{k+1})\right)}{\lambda\left([p_{k}/q_{k},p_{k+1}/q_{k+1})\right)}=0
\implies\liminf_{y\rightarrow x}\frac{Q(x)-Q(y)}{x-y}=0.\]
\end{rem}
For later use, let us also state the following immediate corollary.
\begin{cor}
\label{Corollary4} For $x\in\mathcal{U}$ we have \[
Q'(x)=\infty\,\,\,\hbox{if and only if}\,\,\,\lim_{n\rightarrow
\infty}\nu_{F}(T_{n}(x))/\lambda\left(T_{n}(x)\right)=\infty.\]
\end{cor}
\begin{proof}
The `only if part' of the corollary was obtained in Lemma \ref{Lemma0}.
By noting that the sequence $\left(\nu_{F}\left([p_{k}/q_{k},p_{k+1}/q_{k+1})\right)\,/\,
\lambda\left([p_{k}/q_{k},p_{k+1}/q_{k+1})\right)\right)_{k\in\N}$
is a subsequence of  $\left(\nu_{F}(T_{n}(x))/\lambda\left(T_{n}(x)\right)\right)_{n\in\N}$,
the `if part' of the corollary is an immediate consequence of Proposition
\ref{proposition}. Here, $p_{k}/q_{k}$ refers once more to the $k$-th
convergent of $x$.
\end{proof}

\section{Fractal analysis of the derivative of the Minkowski function}

By Lemma \ref{Lemma1}, the result of \cite{PVB:01} respectively,
the unit interval can be decomposed into pairwise disjoint sets as
follows. \[
\mathcal{U}=\Lambda_{0}\cup\Lambda_{\infty}\cup\Lambda_{\sim},\]
 where \[
\Lambda_{\theta}:=\{ x\in\mathcal{U}:Q'(x)=\theta\}\,\,\hbox{for}\,\,
\theta\in\{0,\infty\},\,\;\;\;\hbox{ and }\,\,
\Lambda_{\sim}:=\mathcal{U}\setminus(\Lambda_{0}\cup\Lambda_{\infty}).\]
 Clearly, by Lemma \ref{Lemma1} we have that
 $\Lambda_{\sim}=\left\{ x\in\mathcal{U}:Q'(x)\,\textrm{does
 not exist and}\,\, Q'(x)\neq\infty\right\} $.\\
 Let us begin our analysis of this decomposition with the following
result.

\begin{prop}
\label{Lemma5} For $s\in(h_{\mathrm{top}},2\log\gamma]$ we have\[
\mathcal{L}(s)\subset\Lambda_{\infty}.\]
 Whereas, for $s\in[0,h_{\mathrm{top}})$ we have \[
\mathcal{L}(s)\subset \Lambda_{0}.\]
\end{prop}
\begin{proof}
Let $x\in\mathcal{L}(s)$ be given. By definition of $\mathcal{L}(s)$,
we then have\[
\lim_{n\rightarrow\infty}\ell_{n}(x)=s.\]
 Hence, for each $\epsilon>0$ there exists $N_{\epsilon}\in\N$ such
that \[
n(s-\epsilon)\leq\log\left(1/\lambda\left(T_{n}(x)\right)\right)\leq n(s+\epsilon),\,\,\,
\textrm{ for all }\,\,\, n\geq N_{\epsilon}.\]
 From this we immediately deduce that \begin{eqnarray}
e^{-n(s+\epsilon-h_{\mathrm{top}})}\leq2^{n}\,\lambda\left(T_{n}(x)\right)\leq e^{-n(s-\epsilon
-h_{\mathrm{top}})},\,\,\,\textrm{ for all }\,\,\, n\geq N_{\epsilon}.\label{est}\end{eqnarray}
 For $s\in(h_{\mathrm{top}},2\log\gamma]$, this implies $\lim_{n\rightarrow\infty}\lambda
 \left(T_{n}(x)\right)/\nu_{F}(T_{n}(x))
 =\lim_{n\rightarrow\infty}2^{n}\,\lambda\left(T_{n}(x)\right)=0$.
By Corollary \ref{Corollary4}, we then have that $Q'(x)=\infty$,
and hence $x\in\Lambda_{\infty}$. This finishes the proof of the
first part of the proposition.\\
For the second part, let $s\in[0,h_{\mathrm{top}})$ and 
 $x=[a_1,a_2,\ldots ]\in\mathcal{L}(s)$
be fixed. Let $q_{n}$ refer  to the denominator
of the $n$-th convergent $p_{n}/q_{n}:=[a_{1},a_{2},\ldots,a_{n}]$ 
of $x$. We then have
\[
\lim_{n \to \infty} \frac{\log(a_{n}q_n q_{n-1})}{\sum_{j=1} ^n a_j}=
\lim_{n \to \infty} \frac{\log(q_n q_{n-1})}{\sum_{j=1} ^n a_j}=
\lim_{n \to \infty} \ell_n(x) <h_{\mathrm{top}}.
\]
Here,  the last equality is a consequence of
\cite[Proposition 2.1]{KesseboehmerStratmann:07}. 
Similar to the above, a straight forward calculation then shows that
$\lim_{n \to \infty} (a_n q_n q_{n-1})/2^{\sum_{j=1}^n a_j}=0$. 
Using the second part of Proposition
\ref{proposition}, it follows $Q'(x)=0$. 
\end{proof}
Note that an immediate consequence of  Proposition \ref{Lemma5} 
is that the essential support of $\nu_F$ is contained
    in $\Lambda_\infty$.
Moreover, by combining  Proposition \ref{Lemma5} and
 Remark \ref{rem5.5} we immediately obtain
the following corollary. Here, $q_{n}$ refers once more
to the denominator
of the $n$-th convergent $p_{n}/q_{n}:=[a_{1},a_{2},\ldots,a_{n}]$
of $x=[a_{1},a_{2},\ldots]$.

\begin{cor}\label{Corollary6} 
For $x\in\mathcal{U}$ the following hold.
\begin{itemize}
\item[(i)] If $\lim_{n\rightarrow\infty}\frac{1}{2\log q_{n}}\sum_{i=1}^{n}a_{i}<1/h_{\mathrm{top}}$,
then $x\in\Lambda_{\infty}$.
\item[(ii)] If $\lim_{n\rightarrow\infty}\frac{1}{2\log q_{n}}\sum_{i=1}^{n}a_{i}>1/h_{\mathrm{top}}$, then $x\in\Lambda_{0}$.
\item[(iii)] If $\limsup_{n\rightarrow\infty}\frac{1}{2\log q_{n}}\sum_{i=1}^{n}a_{i}>1/h_{\mathrm{top}}$
and  $\liminf_{n\rightarrow\infty}\frac{1}{2\log q_{n}}\sum_{i=1}^{n}a_{i}<1/h_{\mathrm{top}}$, then $x\in\Lambda_{\sim}$.
\end{itemize}
\end{cor}

\begin{rem}
Note that a similar type of result was obtained in \cite{PVB:01}.
Namely, on the basis of the assumption that $Q'(x)$ exists in the
generalised sense, the following hold. 
\begin{itemize}
\item[(i)]  \textit{If $\limsup_{n\rightarrow\infty}\frac{1}{n}\sum_{i=1}^{n}a_{i}<2\log\gamma\,\,(=1.3884\ldots)$,
then $x\in\Lambda_{\infty}$. }
\item[(ii)]  \textit{If $\liminf_{n\rightarrow\infty}\frac{1}{n}\sum_{i=1}^{n}a_{i}>\rho\,\,(=5.3197\ldots)$,
then $x\in\Lambda_{0}$.}\\
 \textit{(Here, $\rho$ is given implicitly by $(1+\rho)^{1/\rho}=\sqrt{2}$). }
\end{itemize}
\end{rem}

For the following proposition, let  $N:\mathcal{U} \rightarrow \N$
be given by  $N([a_{1},a_{2},\ldots]):= a_{1}$, and let $I:
\mathcal{U} \rightarrow \R$ refer to the potential function
which is given by  $I(x):=\log\left|G'(x)\right|$, with $G$
denoting the Gauss map. For $0< s	< t	< \infty$, we
then define the sets 
\begin{eqnarray*}
\mathcal{L}^{*}(s)&:=&\left\{ x\in\mathcal{U}:\limsup_{n\rightarrow\infty}\frac{S_{n}I(x)}{S_{n}N(x)} \geq s\right\},\;\;
\mathcal{L}_{*}(s):=\left\{ x\in\mathcal{U}:\liminf_{n\rightarrow\infty}\frac{S_{n}I(x)}{S_{n}N(x)} \geq s\right\},\\
 \mathcal{L}\left(s,t\right)&:=&
\left\{ x\in\mathcal{U}:\liminf_{n \to \infty}\frac{S_{n}I(x)}{S_{n}N(x)}\leq
s ,\limsup_{n \to \infty}\frac{S_{n}I(x)}{S_{n}N(x)}\geq t\right\} ,
\end{eqnarray*}
where $S_{n}\phi(x):= \sum_{k=0}^{n-1} \phi (G^k(x))$ refers to the
$n$-th Birkhoff sum of a function $\phi$. Moreover, for
$x=[a_{1},a_{2},\ldots]\in \mathcal{U}$ and $n \in \N$, we use the
notation $C_{n}(x):= \{[b_{1},b_{2},\ldots]\in \mathcal{U}:
b_{i}=a_{i}, \,\hbox{for all} \, \,  i \in \{1,\ldots,n\}\}$ to denote
the unique $n$-cylinder containing $x$.

\begin{prop}
\label{Lemma6}~
\begin{itemize}
\item [(i)] For each  $  s \in \left[0,2\log\gamma\right] $, we have\[
\dim_{H}(\mathcal{L}_{*}\left(s\right))=\dim_{H}(\mathcal{L}^{*}\left(s\right))
=\dim_{H}(\mathcal{L}\left(s\right)).\]
\item [(ii)]For each $0<s_{0}\leq s_{1}\leq 2\log\gamma$, we have \[
\dim_{H}(\mathcal{L}\left(s_{0},s_{1})\right)=\dim_{H}(\mathcal{L}(s_{1}).\]
\end{itemize}
\end{prop}
\begin{proof}
{\em ad (i).} The inequality $\dim_{H}(\mathcal{L}_{*}
\left(s\right))\leq \dim_{H}(\mathcal{L}^{*}\left(s\right))$ 
follows immediately from $\mathcal{L}_{*}
\left(s\right)\subset\mathcal{L}^{*}\left(s\right)$. 
For the proof of the upper estimate  $\dim_{H}
(\mathcal{L}^{*}\left(s\right))\leq -\widehat{P}\left(-s\right)/s$
we refer to \cite[Lemma~5.4]{KesseboehmerStratmann:07}. Note that
in \cite{KesseboehmerStratmann:07} we in fact considered the set
$\mathcal{L}_{*}(s)$,
rather than the set $\mathcal{L}^{*}(s)$. However, one immediately
sees that in the proof of
\cite[Lemma~5.4]{KesseboehmerStratmann:07}
`$\liminf$' can be replaced by `$\limsup$'. 
Using Theorem \ref{KS} and the fact that $\mathcal{L}\left(s\right)\subset\mathcal{L}_{*}\left(s\right)$, then gives rise
to the statement in (i).

{\em ad (ii).}
Since  "$\leq$" is a 
direct consequence of (i), we only have to show "$\geq$". 
Using standard techniques from geometric measure theory 
(cf. e.g. \cite{M}), it is sufficient to show 
that there exists   a probability measure $\mu$ such that
\begin{itemize}
\item [(A)] $\mu\left(\mathcal{L}\left(s_{0},s_{1}\right)\right)>0$,
\item [(B)] ${\displaystyle \liminf_{n \to \infty}\frac{-\log \mu
\left(C_{n}(x)\right)}{S_{n}I
(x)}\geq\dim_{H}\left(\mathcal{L}\left(s_{1}\right)\right)}$,
for $\mu$-almost every $x \in \mathcal{U}$.
\end{itemize}
For this, let us first recall the following outcome
of the thermodynamic formalism  of \cite{KesseboehmerStratmann:07}.
For $i=0,1$,  let $\mu_{i}$ be the Gibbs measures on $\mathcal{U}$ for
the potential function
$-P\left(t\left(s_{i}\right)\right)I-t\left(s_{i}\right)N$, for
$P$ denoting the pressure function defined in (\ref{pressure}), and $t$ the inverse function of $P'$ (we refer to
\cite[Proposition 4.2]{KesseboehmerStratmann:07} for the details).
For these measures it was shown in \cite{KesseboehmerStratmann:07}  that 
$\int I\, d\mu_{i}/\int N\, d\mu_{i}=s_{i}$,
$h_{\mu_{i}}/\int I\, d\mu_{i}=\dim_{H}\left(\mu_{i}\right)=\dim_{H}
\left(\mathcal{L}\left(s_{i}\right)\right)$,
as well as\begin{equation}
\lim_{n \to \infty} \frac{S_{n}I(x)}{n}=\int I\,
d\mu_{i}\in\left(0,\infty\right)\: \, \textrm{and} \, \,
\lim_{n \to \infty}
\frac{-\log\mu_{i}\left(C_{n}(x)\right)}{n}=h_{\mu_{i}},\;\:\, 
\textrm{ for $\mu_{i}$-almost every $x \in \mathcal{U}$.}\label{eq:ergodprop}\end{equation}
For ease of exposition, let us put $\theta(k):\equiv k\mod(2)$.  
Using Egorov's Theorem, it follows that there exists an increasing 
sequence $\left(m_{k}\right)_{k \in \N}$
and a sequence  $\left(\Gamma_{k}\right)_{
k \in\N}$ of Borel subsets of $\mathcal{U}$, such that  we have 
${\displaystyle \mu_{\theta(k)}\left(\Gamma_{k}\right)\geq1-2^{-(k+1)}}$,
and such that  for all $x \in \Gamma_{k}$
and $n \geq m_{k}$,
\begin{itemize}
\item ${\displaystyle \left|\frac{S_{n}I(x)}{n}-\int I\, d\mu_{\theta(k)}
\right|<k^{-1}}$,
\item ${\displaystyle \left|\frac{-\log\mu_{\theta(k)}
\left(C_{n}(x)\right)}{n}-h_{\mu_{\theta(k)}}\right|<k^{-1}}$,
\item ${\displaystyle \frac{-\log\mu_{\theta(k)}(C_{n}(x))}{S_{n}
I\left(x\right)}>
\dim_{H}\left(\mu_{\theta(k)}\right)-k^{-1}}$.
\end{itemize} Define $n_{0}:=1+1/m_{1}$ and let
$n_{k}:=\prod_{i=1}^{k}\left(m_{i}+1\right)$, for each $k \in \N$.
Then define the countable family of
cylinder sets \[
\mathcal{C}_{k}:=\left\{ C_{n_{k-1}m_{k}}(x):x\in\Gamma_{k}\right\}
,\; \textrm{for each} \, \, k \in \N.\]
 This allows to introduce another family $\left(\mathcal{D}_{k}\right)_{k \in
 \N}$ of cylinder sets as follows.  Let $\mathcal{D}_{1}:=\mathcal{C}_{1}$,
 and for $k \geq2$ define \[
\mathcal{D}_{k}:=\left\{ CD :C\in\mathcal{D}_{k-1},D\in\mathcal{C}_{k}\right\} ,\]
 where $CD$ denotes the concatenation of the cylinders $C$ and $D$.
By construction, we have  that each cylinder set in $\mathcal{D}_{k}$
has length equal to $n_{k}$, for each $k \in\N$.  We can then define the set\[
\mathcal{M}:=\bigcap_{k \in \N}\bigcup_{D \in\mathcal{D}_{k}}D.\]
One immediately verifies that $\mathcal{M}$ is non-empty.
Next, using Kolmogorov's consistency theorem, we define the Cantor-measure $m$ on $\mathcal{M}$, by  setting $m(C):=\mu_{1}(C)$
if $C\in\mathcal{D}_{1}$, and for $C=D'C'\in\mathcal{D}_{k}$ such
that
$D'\in\mathcal{D}_{k-1}$ and $C'\in\mathcal{C}_{k}$, we let \[
m(C):=m(D')\mu_{\theta(k)}(C').\]
 Clearly,  $m$ admits an extension $\mu$ to $\mathcal{U}$, and this is
 given by  $\mu(A):=m(A\cap\mathcal{M})$,
for each  $A\subset\mathcal{U}$ measurable. By construction
we then have that
\[
\mu(\mathcal{M})\geq\prod_{k \in \N}\left(1-2^{-k}\right)>0.\]
Since $I$ is H\"older continuous, we obtain  for $x\in C\in\mathcal{D}_{k}$,
 \begin{eqnarray*}
\left|\frac{S_{n_{k}}I\left(x\right)}{n_{k}}\right| & \leq &
\frac{1}{m_{k}+1}\left|\frac{1}{n_{k-1}}S_{n_{k-1}}I \left(x\right)\right|
+\frac{m_{k}}{m_{k}+1}\left|\frac{1}{n_{k-1}m_{k}}S_{n_{k-1}m_{k}} I
\left(G^{n_{k-1}}x\right)\right|.\end{eqnarray*}
Using this, a straightforward  inductive argument then gives that $S_{n_{k}}
 I (x)/ n_{k}$
is bounded, and  hence, \[
\lim_{k \to \infty} \left|\frac{S_{n_{k}} I \left(x\right)}{n_{k}}-\int  I \, d\mu_{\theta(k)}\right| =0.\]
This shows that $\mathcal{M}\subset\mathcal{L}\left(s_{0},s_{1}\right)$,
and thus the assertion in (A) follows.

For the proof of (B), first note that an  argument  similar to the one
just given,
 shows  \[
\lim_{k \to \infty} \left|\frac{-\log\left(\mu \left(C_{n_{k}}\left(x\right)\right)\right)}
{n_{k}}-h_{\mu_{\theta(k)}}\right|= 0.\]
Then note that $C_{n_{k}}\left(x\right)=C_{n_{k-1}}\left(x\right)C_{m_{k}n_{k-1}}\left(G^{n_{k-1}}x\right)$, 
for each  $x\in\mathcal{M}$
and $k \in\N$. Using this, it follows

 ${\displaystyle \frac{-\log\left(\mu \left(C_{n_{k}}\left(x\right)\right)\right)}{S_{n_{k}}I\left(x\right)}}$
 \begin{eqnarray*}
 & = & \frac{n_{k-1}}{n_{k}}\frac{\frac{S_{n_{k-1}}I\left(x\right)}{n_{k-1}}}{\frac{S_{n_{k}}I\left(x
 \right)}{n_{k}}}\cdot\frac{-\log\left(\mu \left(C_{n_{k-1}}\left(x\right)\right)
 \right)}{n_{k-1}}+\frac{m_{k}n_{k-1}}{n_{k}}\frac{\frac{S_{m_{k}n_{k-1}}I
 \left(G^{n_{k-1}}x\right)}{m_{k}n_{k-1}}}{\frac{S_{n_{k}}I\left(x\right)}{n_{k}}}\cdot\frac{-\log\left(
 \mu_{\theta(k)}\left(C_{m_{k}n_{k-1}}\left(G^{n_{k-1}}x\right)\right)\right)}{S_{m_{k}
 n_{k-1}}I\left(G^{n_{k-1}}x\right)}\\
 & = & \frac{1}{m_{k}+1}\underbrace{\frac{S_{n_{k-1}}I\left(x\right)/n_{k-1}}{S_{n_{k}}
 I\left(x\right)/n_{k}}\cdot\frac{-\log\left(\mu\left(C_{n_{k-1}}
 \left(x\right)\right)\right)}{n_{k-1}}}_{\textrm{bounded}}\\
&& \;\;\;\;\;\;\;\;\;\;\;\;\;\;\;\;\;\;\;\;\;\;\;\;\;\;\;\;\;\;\;\;\;\;\;\;\;\;\;\;\;\;\;\;\;\;\;\;\;\;\;\;\;\;\;\;\;\;\;\;\;\;\;\;\;\;\;
+\frac{m_{k}}{m_{k}+1}
 \underbrace{\frac{\frac{S_{m_{k}n_{k-1}}I\left(G^{n_{k-1}}x\right)}{m_{k}
 n_{k-1}}}{\frac{S_{n_{k}}I\left(x\right)}{n_{k}}}}_{\to1}\cdot\frac{-\log
 \left(\mu_{\theta(k)}\left(C_{m_{k}n_{k-1}}\left(G^{n_{k-1}}x\right)\right)\right)}{S_{m_{k}n_{k-1}}I
 \left(G^{n_{k-1}}x\right)}.\end{eqnarray*}
 This implies that
 \begin{equation}
\liminf_{k \to \infty}\frac{-\log\left(\mu \left(C_{n_{k}}\left(x\right)\right)\right)}{S_{n_{k}}I\left(x\right)}
\geq\dim_{H}\left(\mu_{1}\right).\label{eq:Entropy}\end{equation}
Also,  for
$n_{k}\leq n<n_{k}+m_{k}$
we immediately  obtain \[
\frac{-\log\left(\mu \left(C_{n}\left(x\right)\right)\right)}{S_{n}I\left(x\right)}\geq
\frac{-\log\left(\mu \left(C_{n_{k}}\left(x\right)\right)\right)}{n_{k}}\frac{n_{k}}{n_{k}+m_{k}}.\]
Finally,  if $n_{k}+m_{k} \leq n < n_{k+1}$  then   $C_{n}\left(x\right)=DB$,
for some $D\in \mathcal{D}_{k}$ and for some cylinder set $B$ of
length at least $m_{k}$ such that  $B$ contains some cylinder set $C
\in \mathcal{C}_{k+1}$.  We then have by
construction that $\mu \left(DC\right)\leq \mu
\left(D\right)\mu_{\theta(k+1)}\left(C\right)$. Using this, it
follows that for each $\epsilon>0$ and $n$ sufficiently large,\begin{eqnarray*}
\frac{-\log\left(\mu\left(C_{n}\left(x\right)\right)\right)}{S_{n}I\left(x\right)} & \geq & \frac{-\log\left(\mu
\left(C_{n_{k}}\left(x\right)\right)\right)-\log\mu_{\theta(k+1)}\left(C_{\left|B\right|}\left(G^{n_{k}}
\left(x\right)\right)\right)}{S_{n}I\left(
x\right)}\\
 & \geq & \frac{\left(\dim_{H}\left(\mu_{1}\right)-\epsilon\right)S_{n_{k}}I\left(x\right)+
 \left(\dim_{H}\left(\mu_{1}\right)-\epsilon\right)S_{\left|B\right|}I\left(G^{n_{k}}\left(x\right)
 \right)}{S_{n}I\left(x\right)}\\
 & = & \dim_{H}\left(\mu_{1}\right)-\epsilon.\end{eqnarray*}
 By combining the two latter inequalities,  the assertion in  (B) follows.
\end{proof}
\begin{rem}
Note that the proof of Proposition \ref{Lemma6} (ii) was inspired by the  argument  in
\cite[Theorem 6.7 (3)]{BS00}. However, the considerations in \cite{BS00} are restricted
to expanding dynamical systems, whereas the dynamical system in
Proposition \ref{Lemma6} is expansive. Hence, the proof of Proposition \ref{Lemma6} (ii) 
can be considered as giving a partial extension of the result in \cite{BS00}.
\end{rem}

The following theorem gives the main result of this paper.
\begin{thm}
\label{Theorem} For the Hausdorff dimensions of $\Lambda_{\infty}$
and $\Lambda_{\sim}$ we have \[
\dim_{H}(\Lambda_{\sim})=\dim_{H}\left(\Lambda_{\infty}\right)=\dim_{H}\left(\mathcal{L}(h_{\mathrm{top}})\right).\]
\end{thm}
\begin{rem}
By combining Theorem \ref{Theorem}, Proposition \ref{Lemma:Kinney}
and Remark \ref{remarkK}, and using the fact that $h_{\mathrm{top}}<\chi_{\nu_{F}}\approx0.792$,
one immediately finds that the actual value of $\dim_{H}\left(\mathcal{L}(h_{\mathrm{top}})\right)$
is trapped between $1$ and the Hausdorff dimension of the measure
of maximal entropy of the Farey map (cf. Figure \ref{fig:Spectrum}). That is, we have \[
0.875\approx\dim_{H}(\mathcal{L}(\chi_{\nu_{F}}))=\dim_{H}(\nu_{F})<\dim_{H}
\left(\mathcal{L}(h_{\mathrm{top}})\right)<\dim_{H}\left(\mathcal{L}(0)\right)=1.\]
\end{rem}
\begin{proof}
For the proof of the second equality in the theorem, it is sufficient
to show that \begin{eqnarray}
\mathcal{L}(h_{\mathrm{top}}+\kappa)\subset\Lambda_{\infty}\subset\mathcal{L}_*(h_{\mathrm{top}}),\,\,
\textrm{ for each }\,\,\kappa>0.\label{inc}\end{eqnarray}
The first inclusion is just the first statement in Proposition \ref{Lemma5}. For the second inclusion in (\ref{inc}), let $x\in\Lambda_{\infty}$ be
given. We then have $\lim_{n\rightarrow\infty}2^{n}\,\lambda\left(T_{n}(x)\right)=0$,
which gives that for each $\epsilon>0$ there exists $N_{\epsilon}\in\N$
such that $2^{n}\,\lambda\left(T_{n}(x)\right)<\epsilon$, for all
$n\geq N_{\epsilon}$. Now note that we have the following chain of
implications. \begin{eqnarray*}
2^{n}\,\lambda\left(T_{n}(x)\right)<\epsilon & \Longrightarrow & \lambda\left(T_{n}(x)\right)<\epsilon\,2^{-n}
\Longrightarrow\,\log\lambda\left(T_{n}(x)\right)<-nh_{\mathrm{top}}+\log\epsilon\\
 & \Longrightarrow & \ell_{n}(x)>h_{\mathrm{top}}-\log\epsilon/n.\end{eqnarray*}
 It follows that $\liminf_{n\rightarrow\infty}\frac{S_n I(x)}{S_n N(x)}\geq\liminf_{n\rightarrow\infty}\ell_{n}(x)\geq h_{\mathrm{top}}$.
This shows that $x\in\mathcal{L}_*(h_{\mathrm{top}})$, and hence,
$\Lambda_{\infty}\subset\mathcal{L}_*(h_{\mathrm{top}})$. This finishes
the proof of the second inclusion in (\ref{inc}), and hence finishes
the proof
of the second equality stated in the theorem.\\
For the remaining assertions of the theorem, first note that by 
Lemma \ref{Lemma6} we have 
$\dim_{H}(\mathcal{L}(h_{\mathrm{top}}))=
\dim_{H}(\mathcal{L}^{*}(h_{\mathrm{top}}))$.
 Hence, for the upper bound, it is sufficient to show that $\Lambda_{\sim}\subset\mathcal{L}^{*}(h_{\mathrm{top}})$.
In order to prove this, note that we have that  $\Lambda_{\sim}\subset \mathcal{U}\setminus\Lambda_0$.
By the second part of Proposition \ref{proposition} we have 
\[x\in\mathcal{U}\setminus\Lambda_0 \implies \limsup_{n \to \infty} \frac{a_nq_nq_{n-1}}{2^{\sum_{j=1}^n a_j}}>0
 \implies \limsup_{n \to \infty} \frac{S_nI(x)}{S_nN(x)}\geq h_\mathrm{top},
\] 
and hence $x\in\mathcal{L}^{*}(h_{\mathrm{top}})$.
This finishes the proof of the upper bound $\dim_{H}(\Lambda_{\sim})\leq\dim_{H}\left(\mathcal{L}(h_{\mathrm{top}})\right)$.\\
 For the lower bound, note that by Corollary \ref{Corollary6} we have
 that \[
\left\{ x\in\mathcal{U}:\liminf_{n \to \infty}
\frac{S_nI(x)}{S_nN(x)}<h_{\mathrm{top}}<\limsup_{n \to \infty}
\frac{S_nI(x)}{S_nN(x)}\right\} \subset\Lambda_{\sim}.\]
 Hence, it is sufficient to show that
 $\dim_{H}\left(\mathcal{L}\left(s_{0},s_{1}\right)\right)\geq
 \dim_{H}\left(\mathcal{L}\left(s_{1}\right)\right)$,
for  each $s_{0}\in\left(0,h_{\mathrm{top}}
 \right)$ and $s_{1}\in\left(h_{\mathrm{top}},\infty\right)$.
 Since the latter is an immediate consequence of Proposition \ref{Lemma6},
 the proof of the theorem is complete.
\end{proof}
Let us finish the paper with the following
immediate consequence of Theorem \ref{Theorem}.
\begin{cor}
For the Hausdorff dimension of $\mathcal{U}\setminus\Lambda_{0}$
we have \[
\dim_{H}(\nu_{F})<\dim_{H}(\mathcal{U}\setminus\Lambda_{0})=\dim_{H}(\mathcal{L}(h_{\mathrm{top}}))<1.\]
 In particular, this implies the aforementioned result of
Salem \cite{Salem}, namely that $Q$ is a singular function in the sense that \[
\lambda\left(\Lambda_{0}\right)=1.\]
\end{cor}

\end{document}